\documentclass[11pt,a4paper]{amsart}

\usepackage{amssymb}
\usepackage[final]{showkeys}
\usepackage{microtype}
\usepackage{color}
\usepackage{graphicx}
\usepackage[hmargin=3cm,vmargin={3.5cm,4cm}]{geometry}

\newtheorem{te}{Theorem}[section]

\newtheorem{os}[te]{Remark}
\newtheorem{prop}[te]{Proposition}

\newtheorem{coro}[te]{Corollary}

\numberwithin{equation}{section}

\allowdisplaybreaks

\begin{document}

    \title[]{Random motions with space-varying velocities}

    \author{Roberto Garra$^1$}

    \author{Enzo Orsingher$^1$}
    \address{${}^1$Dipartimento di Scienze Statistiche, ``Sapienza'' Universit\`a di Roma.}


    \date{\today}

    \begin{abstract}
    Random motions on the line and on the plane with space-varying velocities are considered and analyzed in this paper.
    On the line we investigate symmetric and asymmetric telegraph processes with space-dependent velocities 
    and we are able to present the explicit distribution of the position $\mathcal{T}(t)$, $t>0$, of the moving particle.
    Also the case of a non-homogeneous Poisson process (with rate $\lambda = \lambda(t)$) governing the changes of direction is
    analyzed in three specific cases. For the special case $\lambda(t)= \alpha/t$ we obtain a random motion related to the Euler-Poisson-Darboux (EPD) equation
    which generalizes the well-known case treated e.g. in \cite{foong}, \cite{markov} and \cite{Rosen}. 
    A EPD--type fractional equation is also considered and a parabolic solution (which in dimension $d=1$ has the structure of a probability density)
    is obtained.\\
    Planar random motions with space--varying velocities and infinite directions are finally analyzed in Section 5. We are able to present 
    their explicit distributions and for polynomial-type velocity structures we obtain the hyper and hypo-elliptic form of their support
    (of which we provide a picture).
    
	\smallskip

    \end{abstract}

    \maketitle

    \section{Introduction}
    
    The telegraph process represents a simple prototype of finite velocity random motions on the
    line, whose probability law is governed by a hyperbolic partial differential equation that is the classical telegraph equation, widely
    used in mathematical physics both in problems of electromagnetism and heat conduction (see for example \cite{cat}). 
    In \cite{masoliver}, the authors studied a generalization of the classical telegraph process with space-time varying propagation speed.
    Within this framework, the probabilistic model is based on the limit of a persistent random walk on a non--uniform lattice. The consequence of the assumption of a space-time depending velocity $c(x,t)$ is that the probability law of the corresponding finite velocity random motion is 
    governed by the following telegraph equation with variable coefficients
    \begin{equation}\label{1}
    \frac{\partial}{\partial t}\left[\frac{1}{c(x,t)}\frac{\partial p}{\partial t}\right]+2\lambda \frac{1}{c(x,t)}\frac{\partial p}{\partial t}
    = \frac{\partial}{\partial x}\left[c(x,t)\frac{\partial p}{\partial x}\right].
    \end{equation}
    In some cases it is possible to find the explicit form of the probability law of this generalization of the telegraph process, by solving 
    equation \eqref{1} subject to suitable initial conditions. In particular, we focus our attention on the case of space--depending velocity,
    where \eqref{1} becomes
     \begin{equation}\label{2}
        \frac{\partial^2 p}{\partial t^2}+2\lambda\frac{\partial p}{\partial t}
        = c(x)\frac{\partial}{\partial x}\left[c(x)\frac{\partial p}{\partial x}\right].
        \end{equation}
      Equation \eqref{2} can be reduced to the form of the classical 
      telegraph equation by using the change of variable
      \begin{equation}
      y =\int_0^x \frac{dx'}{c(x')}, \quad x\in \mathbb{R}
      \end{equation}
      provided that $c(x)>0$, $x\in \mathbb{R}$ such that $\int_0^x\frac{dw}{c(w)}<\infty$ for all $x\in \mathbb{R}$.
      Therefore in this case the explicit probability law is simply given 
      by
      \begin{align}
      \nonumber & p(x,t)= \frac{e^{-\lambda t}}{2}\bigg\{\delta\left(t-\bigg|\int_0^x\frac{dx'}{c(x')}\bigg|\right)+\delta\left(t+\bigg|\int_0^x\frac{dx'}{c(x')}\bigg|\right)\bigg\}\\
      \nonumber &+\frac{e^{-\lambda t}}{2c(x)}\bigg[\lambda I_0\bigg(\lambda \sqrt{t^2-\bigg|\int_0^x\frac{dx'}{c(x')}\bigg|^2}\bigg)
      +\frac{\partial}{\partial t} I_0\bigg(\lambda \sqrt{t^2-\bigg|\int_0^x\frac{dx'}{c(x')}\bigg|^2}\bigg)\bigg] \times \mathbf{1}_{D}(x)
      \bigg\},
      \end{align}
	where $\mathbf{1}_D$ is the characteristic function of the set 
	$$D:= \bigg\{x\in \mathbb{R}:\bigg|\left(\int_0^x\frac{dx'}{c(x')}\right)\bigg|<t\bigg\}$$
    and $I_0(\cdot)$ is the modified Bessel function of order zero. \\
    Taking for example $c(x)= |x|^\alpha$, the endpoints of the domain $D$ are given by $\pm[(1-\alpha)t]^{\frac{1}{1-\alpha}}$ for $\alpha<1$ and become $\pm \infty$ for $\alpha\geq 1$. \\
    In this paper we consider the asymmetric telegraph process with space-varying velocity and also the symmetric telegraph
    process with a non-homogeneous Poisson process governing the changes of space-dependent velocities.\\
    A section is devoted to a fractional Euler-Poisson-Darboux-type equation and to the discussion of a special class 
    of non-negative solutions.
	
	While the telegraph process on the line is essentially a persistent random walk with only two possible directions, the picture of finite velocity random motions on the plane and in the space is more complicated and gives rise to the studies of random flights (see for example
	\cite{ale1,ale2,bruno,pogorui}).
	An interesting result, in this context, was proved by Kolesnik and Orsingher in \cite{kole}, where the connection between planar random motions with an infinite number of possible directions and the damped wave equation was discussed. In their model, the motion is described by
	a particle taking directions $\theta_j$ , $j = 1, 2,...$, uniformly
	distributed in $[0; 2\pi)$ at Poisson paced times. The orientations
	$\theta_j$ are i.i.d. r.v.'s independent from the homogeneous Poisson process 
	$N(t)$ of rate $\lambda$ governing the changes of direction. The particle
	starts off at time $t=0$ from the origin and moves with constant velocity $c$. 
	At the epochs of the Poisson process the particle takes new directions (uniformly distributed in $[0,2\pi)$),
	independent from its previous evolution.
	Under these assumptions, it is possible to prove that the explicit probability law of the current position $(X(t),Y(t))$ of the randomly moving particle is a solution of the damped wave equation
	\begin{equation}\label{1.4}
	\frac{\partial^2 p}{\partial t^2}+2\lambda \frac{\partial p}{\partial t}= c^2 \left[\frac{\partial^2 p}{\partial x^2}+\frac{\partial^2 p}{\partial y^2}\right].
	\end{equation}   
	
	\smallskip
	
	In the last part of this paper, we consider the effect of a space-varying speed of propagation on the model of planar random motions with infinite possible directions, leading to the equation
	\begin{equation}\label{1.5}
	\frac{\partial^2 p}{\partial t^2}+2\lambda \frac{\partial p}{\partial t}= c_1(x) \frac{\partial}{\partial x}\left(c_1(x)\frac{\partial p}{\partial x}\right)+c_2(y) \frac{\partial}{\partial y}\left(c_2(y)\frac{\partial p}{\partial y}\right).
	\end{equation}
	 We show the consequence of assuming space-varying velocities on the form of the support $\mathcal{D}$ of the distribution of $(X(t),Y(t))$. By means of the transformation 
	 \begin{align}
	 \nonumber & u = \int_0^x \frac{dw}{c_1(w)}\\
	 \nonumber & v = \int_0^y \frac{dz}{c_2(z)},
	 \end{align}
	 the equation \eqref{1.5} is reduced to the form \eqref{1.4} and thus
	 we can obtain the explicit distribution $p(x,y,t)$ of $(X(t),Y(t))$.
	 We then examine the form of the support of $p = p(x,y,t)$ and analyze
	 its dependence on the space-varying velocity.
	 
	In the special case where $c_1(x)= |x|^\gamma/c_1$, $c_2(y)= |y|^\beta/c_2$, $\gamma,\beta < 1$, we obtain that the boundary of $\mathcal{D}$ is hyper-elliptic for $\gamma = \beta <0$ and
	hypo-elliptic for $\gamma = \beta >0$ and elliptic for $\gamma = \beta = 0$.
	
	\section{Telegraph process with drift and space--varying velocity}
	
	In this section we consider a generalization of the telegraph process with drift considered by Beghin et al. (see reference 
	\cite{luciano}) in the case where the velocity is assumed to be space--varying. In particular, here we consider the random motion of a particle moving on the line and switching from  
	the space-varying (positive) velocity $c(x)$ to $-c(x)$ after an exponentially distributed time with rate $\lambda_1$ and from $-c(x)$ to $c(x)$ after an exponential time with a different rate $\lambda_2$.
	For the description of the random position of the particle $X(t)$ at time $t>0$ we use the following probability densities
	\begin{equation}
	\begin{cases}
	f(x,t)dx=P\{X(t)\in dx, V(t)= c(x) \}\\
	b(x,t)dx= P\{X(t)\in dx, V(t)= -c(x) \},
	\end{cases}
	\end{equation}
	satisfying the system of partial differential equations (see \cite{e85} for a detailed probabilistic derivation)
		\begin{equation}
		\begin{cases}
		\displaystyle \frac{\partial f}{\partial t}=-c(x)\frac{\partial f}{\partial x}-\lambda_1 f +\lambda_2 b\\
		\displaystyle\frac{\partial b}{\partial t}=c(x)\frac{\partial b}{\partial x}+\lambda_1 f -\lambda_2 b.
		\end{cases}
		\end{equation}
		Defining 
		\begin{equation}
		p(x,t)= f+b, \quad w = f-b,
		\end{equation}
		we have the following system of equations
		\begin{equation}
				\begin{cases}
			\displaystyle	\frac{\partial p}{\partial t}=-c(x)\frac{\partial w}{\partial x}\\
			\displaystyle	\frac{\partial w}{\partial t}=-c(x)\frac{\partial p}{\partial x}+\lambda_2 (p-w) -\lambda_1 (p+w).
				\end{cases}
				\end{equation}
		Therefore the probability law $p(x,t)$ is governed by
		the following telegraph-type equation with space-varying velocity and drift
		\begin{equation}\label{rel}
		\frac{\partial^2 p}{\partial t^2}+(\lambda_1+\lambda_2)\frac{\partial p}{\partial t} = 
		c(x)\frac{\partial}{\partial x} (c(x)\frac{\partial p}{\partial x})+c(x)(\lambda_1-\lambda_2)\frac{\partial p}{\partial x}.
		\end{equation}
		In order to eliminate the drift term and to find the explicit form of the probability law, we now introduce the following
		Lorentz-type transformation of variables
		\begin{equation}\label{lor}
		\begin{cases}
		x'= A\displaystyle\int_0^x\frac{dw}{c(w)}+Bt\\
		t' = C\displaystyle\int_0^x\frac{dw}{c(w)}+Dt.
		\end{cases}
		\end{equation}
		By means of some calculation we obtain that, by taking the following choice of the coefficients appearing in \eqref{lor}
		\begin{equation}
		A= D= 1, \quad B= C = \frac{\lambda_1-\lambda_2}{\lambda_1+\lambda_2},
		\end{equation}
		equation \eqref{rel} becomes the classical telegraph equation 
		\begin{equation}
		\frac{\partial^2 p}{\partial t'^2}+(\lambda_1+\lambda_2)\frac{\partial p}{\partial t'}= \frac{\partial^2 p}{\partial x'^2}
		\end{equation}
		and we can therefore find the explicit probability law, starting from that of the classical telegraph process.
	
	\section{Non-homogeneous telegraph processes with space--varying velocities}
	
	Let us recall that a telegraph process $\mathcal{T}(t)$, $t>0$, where changes of direction are paced by a non-homogeneous Poisson process, denoted by $\mathcal{N}(t) $, with time-dependent rate $\lambda(t)$, $t>0$, has distribution $p(x,t)$ satisfying
		the Cauchy problem (see e.g. \cite{kaplan})
		\begin{equation}\label{sec3}
		\begin{cases}
		\displaystyle\frac{\partial^2 p}{\partial t^2}+2\lambda(t)\frac{\partial p}{\partial t}= c^2\frac{\partial^2 p}{\partial x^2},\\
		\displaystyle p(x,0)=\delta(x),\quad \frac{\partial p}{\partial t}(x,t)\big|_{t=0}=0.
		\end{cases}
		\end{equation}
	     In order to obtain explicit distributions in some specific cases, we observe that, by means of the exponential transformation
			\begin{equation}\label{exp}
			p(x,t)= e^{-\int_0^t\lambda(s)ds} v(x,t),
			\end{equation}
			we convert the equation in \eqref{sec3} into
			\begin{equation}\label{KG}
			\frac{\partial^2 v}{\partial t^2}-[\lambda'(t)+\lambda^2(t)]v= c^2\frac{\partial^2 v}{\partial x^2}.
			\end{equation}
			Then, in order to find the explicit probability law of $\mathcal{T}(t)$ from \eqref{sec3}, a mathematical trick is to solve the following 
			Riccati equation emerging from \eqref{KG} (see \cite{markov} and \cite{Iacus})
			\begin{equation}\label{Riccati}
			\lambda'(t)+\lambda^2(t)= const. 
			\end{equation}
		    In this way, it is possible to find, in particular, the following probability laws with absolutely continuous components given by
		    \begin{equation}
		    	P\bigg\{Y(t) \in dx\bigg\}/dx=\frac{1}{2c \, \cosh \lambda t}\frac{\partial}{\partial t}I_0\left(\frac{\lambda}{c}\sqrt{c^2t^2-x^2}\right)
		    	    \label{ras}, \quad |x|<ct,
		    	\end{equation}
		    and
		    \begin{equation}\label{pro1}
		    	   P\bigg\{X(t)\in dx\bigg\}/dx=\frac{\lambda I_0\left(\frac{\lambda}{c}\sqrt{c^2t^2-x^2}\right)}{2c\sinh \lambda
		    	     t }, \quad |x|<ct,
		    	   \end{equation}
		    corresponding to the cases
		    	\begin{equation*}
		    	\begin{cases}
		    	\lambda(t)= \lambda\tanh \lambda t,\\
		    	\lambda(t)= \lambda \coth \lambda t,
		    	\end{cases}
		    	\end{equation*} 
		    	respectively. We observe that the process $Y(t)$ has a discrete component of the distribution concentrated at $x = \pm ct$ (see \cite{Iacus}), while $X(t)$ has only an 
		    	absolutely continuous distribution (see \cite{markov}).
		    	
		     Starting from \eqref{ras} and \eqref{pro1}, we can clearly build other families of explicit probability laws of the form
		    \begin{equation}
		    \begin{cases}
		    	P\bigg\{Y(t) \in dx\bigg\}/dx=\frac{1}{\displaystyle 2 c(x) \, \cosh \lambda t}\frac{\partial}{\partial t}I_0\left(\lambda\sqrt{t^2-\displaystyle\bigg|\int_0^x\frac{dx'}{c(x')}\bigg|^2}\right),  \\
		    	 P\bigg\{X(t)\in dx\bigg\}/dx=\frac{\lambda}{\displaystyle 2c(x) \sinh \lambda t }I_0\left(\lambda\sqrt{t^2-\displaystyle\bigg|\int_0^x\frac{dx'}{c(x')}\bigg|^2}\right)	 ,  \quad \mbox{for $\displaystyle \bigg\{x:\bigg|\int_0^x\frac{dx'}{c(x')}\bigg|<t\bigg\}$,} 	    
		    \end{cases}
		    \end{equation}
		    which depend on the particular choice of $c(x)>0$ (s.t. $\int_0^\infty \frac{dw}{c(w)}<\infty$ for all $x$). These probability laws are clearly related to the following partial differential equations
		    \begin{equation}
		    \begin{cases}
		    &\displaystyle{\frac{\partial^2 p}{\partial t^2}+2\lambda\tanh \lambda t\frac{\partial p}{\partial t}= c(x)\frac{\partial}{\partial x}c(x)\frac{\partial p}{\partial x}},\\
		    &\displaystyle{\frac{\partial^2 p}{\partial t^2}+2\lambda \coth\lambda t\frac{\partial p}{\partial t}= c(x)\frac{\partial}{\partial x}c(x)\frac{\partial p}{\partial x}},\\
		    \end{cases}
		    \end{equation}
			respectively.
			
			Another interesting case is $\lambda(t)=\frac{\alpha}{t}$, which converts equation \eqref{sec3} into the classical Euler-Poisson-Darboux equation.

				 The first probabilistic interpretation of the fundamental solution of the EPD equation was given by Rosencrans in \cite{Rosen} and some of its generalizations have been considered in \cite{markov}.
				In the spirit of the previous observations, we have that
				the solution of the Cauchy problem 
				\begin{equation}
				\begin{cases}
				\displaystyle\frac{\partial^2 v}{\partial t^2}+\frac{2\alpha}{t}\frac{\partial v}{\partial t}= 
				c(x)\frac{\partial}{\partial x}c(x)\frac{\partial v}{\partial x},\\
				v(x,0)=\delta(x),\\
				\frac{\partial v}{\partial t}\bigg|_{t=0}=0
				\end{cases}
				\end{equation}
				can be written as
				  \begin{equation}\label{fon}
				  v(x,t) =  \frac{1}{B(\alpha, \frac{1}{2})\, c(x)t}\left(1-\frac{\bigg|\displaystyle\int_0^x\frac{dx'}{c(x')}\bigg|^2}{t^2}\right)^{\alpha-1},
				     \quad \mbox{for $\bigg\{\displaystyle x:\bigg|\int_0^x\frac{dw}{c(w)}\bigg|<t\bigg\}$.}
				     \end{equation}

		We finally observe that it is possible to consider other cases of non-homogeneous telegraph processes with space-dependent velocities according to the following simple steps:
		\begin{itemize}
		\item Consider the equation
		\begin{equation}
		\frac{\partial^2 p}{\partial t^2}+2\lambda(t)\frac{\partial p}{\partial t}= c(x)\frac{\partial}{\partial x}c(x)\frac{\partial p}{\partial x},
		\end{equation} 
		governing a telegraph process on the line, where the changes of direction are given by a non-homogeneous Poisson process with a deterministic time-dependent rate $\lambda(t)$ and with space-dependent velocity $c(x)$. 
		\item Define the new variables $x'= \int_0^x\frac{du}{c(u)}$ and $t'= \int_0^t \gamma(s)ds$, where $\gamma(t)$ is a $C^1[0,+\infty)$ function that will be defined in the next step;
		\item In the new variables we have that $p(x',t')$ satisfies the equation
		\begin{equation}
		\gamma^2(t')\frac{\partial^2 p}{\partial t'^2}+\left(\gamma'+2\lambda\gamma\right)\frac{\partial p}{\partial t'}= \frac{\partial^2p}{\partial x'^2};
		\end{equation}
		\item Take $\gamma(t)$ such that $\frac{\gamma'}{\gamma}= -2\lambda(t)$. Then the problem is finally reduced to the following D'Alembert equation with time-depending coefficient
		\begin{equation}\label{daa}
		\frac{\partial^2 p}{\partial t'^2}= \frac{1}{\gamma^2(t')}\frac{\partial^2 p}{\partial x'^2}.
		\end{equation}
		\item By taking the further change of variable $(x',t')\rightarrow (\gamma(t')x',t')$ and calling $x''=\gamma(t')x'$ we finally reduce
		equation \eqref{daa} to the classical D'Alembert equation in the variables $(x'',t')$
		\begin{equation}
		\frac{\partial^2 u}{\partial t'^2}= \frac{\partial^2 u}{\partial x''^2}.
		\end{equation} 
		Thus an observer in the framework $(x'',t')$ sees the original random motion transformed into a deterministic one governed by the classical D'Alembert equation.
		\end{itemize}

	\section{Time-fractional Euler-Poisson-Darboux equation with variable velocity}
	
	We here provide some new results about the Euler-Poisson-Darboux equation involving time-fractional derivatives in the sense of Riemann-Liouville (see \cite{Kilbas}) and with space-varying velocity. 
	It is well-known that the EPD equation governs a telegraph process with time-dependent rate $\lambda(t)= \alpha/t$. As far as we know this is the first investigation about the time-fractional EPD equation.
	
	\begin{te}
	The $d$-dimensional time fractional EPD-type equation
	\begin{equation}\label{epd}
	\left(\frac{\partial^{2\nu}}{\partial t^{2\nu}}+\frac{C_1}{t^\nu}\frac{\partial^\nu}{\partial t^\nu}\right)u = \Delta u,
	\end{equation}
	with $\nu \in (0,1)\setminus\{\frac{1}{2}, \frac{1}{3},\frac{1}{4}, \frac{1}{5}\}$ and
	\begin{equation}
	C_1 = -\frac{\Gamma(1-4\nu)}{\Gamma(1-5\nu)},
	\end{equation}
	admits the following non-negative solution:\\
	for $C_2>0$
	\begin{equation}\label{solutio}
	u(\mathbf{x}_d,t)= \begin{cases}
	\displaystyle\frac{1}{t^\nu}\bigg[1-C_2\frac{ \|\mathbf{x}_d\|^2}{t^{2\nu}}\bigg], \quad & \|\mathbf{x}_d\|< \frac{t^{\nu}}{C_2^{1/2}}, \\
	0 \quad &\mbox{elsewhere},
	\end{cases}
	\end{equation}
	while for $C_2<0$
	\begin{equation}\label{solutio2}
		u(\mathbf{x}_d,t)=
		\frac{1}{t^\nu}\bigg[1-C_2\frac{ \|\mathbf{x}_d\|^2}{t^{2\nu}}\bigg], \quad 	\forall \ \mathbf{x}_d\in \mathbb{R}^d
		\end{equation}
    where $\mathbf{x}_d= (x_1,x_2, \dots,x_d)$, 
	$d\in \mathbb{N}$ and 
	$$C_2= -\frac{1}{2d}\bigg[\frac{\Gamma(1-\nu)}{\Gamma(1-3\nu)}
	-\frac{\Gamma(1-4\nu)}{\Gamma(1-5\nu)}\frac{\Gamma(1-\nu)}{\Gamma(1-2\nu)}\bigg]$$
	\end{te}
	\begin{proof}
	By considering that \eqref{epd} has the structure of an EPD equation, we determine a parabolic-type solution of it. By using the well-known fact that (see \cite{Kilbas},
	pag.71)
	\begin{equation}
	\frac{\partial^\alpha t^\beta}{\partial t^\alpha}= \frac{\Gamma(\beta+1)t^{\beta-\alpha}}{\Gamma(\beta+1-\alpha)}, \quad \mbox{for $\alpha >0$ and $\beta>-1$},
	\end{equation}
	we can calculate the exact form of the coefficient $C_2$ such that \eqref{solutio} is a solution of \eqref{epd}. We assume that $\nu\neq \frac{1}{2}, \frac{1}{3},\frac{1}{4}, \frac{1}{5}$ in order to avoid the singularities in the coefficients appearing in $C_1$ and $C_2$. 
	\end{proof}
	
	\begin{os}
	It is possible to construct a probability law with compact support, starting from the general Theorem 4.1 in the one dimensional case, assuming that $\nu$ is such that $C_2$ is positive. In this case we have that the probability law  
		\begin{equation}\label{sol}
		p(x,t)= \frac{N}{t^\nu}\bigg[1-C_2\frac{|x|^2}{t^{2\nu}}\bigg], \quad |x|< \frac{t^{\nu}}{C_2^{1/2}},
		\end{equation}
	    with 
		$$C_2= -\frac{1}{2}\bigg[\frac{\Gamma(1-\nu)}{\Gamma(1-3\nu)}
		-\frac{\Gamma(1-4\nu)}{\Gamma(1-5\nu)}\frac{\Gamma(1-\nu)}{\Gamma(1-2\nu)}\bigg]$$ and 
		$N= \frac{3}{4}\sqrt{C_2}$ the normalizing constant, satisfies the one dimensional time-fractional EPD-type equation \eqref{epd}.\\
		We remark that it is not trivial matter to find the explicit values of $\nu \in (0,1)$ such that the coefficient $C_2>0$.
	\end{os}
	
	Notice that it is extremely hard to ascertain that functions of the form 
	\begin{equation}
	u(x,t)= \frac{N}{t^\beta}\left(1-\frac{ \|\mathbf{x}_d\|^2}{t^\alpha}\right)^\gamma,
	\end{equation}
	are solutions of \eqref{epd} for $\gamma \neq 1$ and suitable $\beta$ and $\alpha$.
	
	\bigskip
	
	We can also observe, with the following Proposition, that we are able to find a solution for a time-fractional EPD-type equation of higher order.
	
	\begin{prop}
	The $d$-dimensional time fractional EPD-type equation
	\begin{equation}\label{epdho}
	\left(\frac{\partial^{2\nu}}{\partial t^{2\nu}}+\frac{C_1}{t^\nu}\frac{\partial^\nu}{\partial t^\nu}\right)u = \sum_{j=1}^d\frac{\partial^{2n}u}{\partial x_j^{2n}}, \quad n\in \mathbb{N},
	\end{equation}
	with $\nu \in (0,1)\setminus\{\frac{1}{2}, \frac{1}{3},\frac{1}{4}, \frac{1}{5}\}$ and
	\begin{equation}
	C_1 = -\frac{\Gamma(1-4\nu)}{\Gamma(1-5\nu)},
	\end{equation}
	admits the following non-negative solution:\\
	for $C_2>0$
	\begin{equation}
	u(x_1, \dots, x_d,t)= \begin{cases}\frac{1}{t^\nu}\bigg[1-C_2\frac{\sum_{j=1}^d x_j^{2n}}{t^{2\nu}}\bigg], \quad & \displaystyle \sum_{j=1}^d x_j^{2n}<\frac{t^{2\nu}}{C_2},\\
	0\quad & \mbox{elsewhere}
	\end{cases}
	\end{equation}
	and for $C_2<0$
	\begin{equation}
	u(x_1, \dots, x_d,t)= \displaystyle \displaystyle \frac{1}{t^\nu}\bigg[1-C_2\frac{\sum_{j=1}^d x_j^{2n}}{t^{2\nu}}\bigg], \quad \forall \ \mathbf{x}_d\in \mathbb{R}^d
    \end{equation}
    where $d\in \mathbb{N}$ and 
	$$C_2= -\frac{1}{(2n)!d}\bigg[\frac{\Gamma(1-\nu)}{\Gamma(1-3\nu)}
	-\frac{\Gamma(1-4\nu)}{\Gamma(1-5\nu)}\frac{\Gamma(1-\nu)}{\Gamma(1-2\nu)}\bigg].$$
	\end{prop} 
	
		Starting from \eqref{sol}, we have the following Corollary.
		\begin{coro}
		Taking $\nu \in (0,1)$ such that $C_2 >0$, the probability law 
		\begin{equation}\label{soluz1}
				p(x,t)= \frac{N}{c(x)t^\nu}\bigg[1-C_2\frac{\bigg|\displaystyle\int_0^x\frac{dx'}{c(x')}\bigg|^2}{t^{2\nu}}\bigg], \quad \mbox{for $\bigg\{x:\bigg|\displaystyle \int_0^x\frac{dx'}{c(x')}\bigg|< \frac{t^{\nu}}{C_2^{1/2}}\bigg\}$},
				\end{equation}
		with $N= \frac{3}{4}\sqrt{C_2}$ and
		 $$C_2= -\frac{1}{2}\bigg[\frac{\Gamma(1-\nu)}{\Gamma(1-3\nu)}
				-\frac{\Gamma(1-4\nu)}{\Gamma(1-5\nu)}\frac{\Gamma(1-\nu)}{\Gamma(1-2\nu)}\bigg],$$ 
		satisfies the time-fractional EPD-type equation with non-constant coefficients
		\begin{equation}
			\left(\frac{\partial^{2\nu}}{\partial t^{2\nu}}+\frac{C_1}{t^\nu}\frac{\partial^\nu}{\partial t^\nu}\right)p = c(x) \frac{\partial}{\partial x}c(x)\frac{\partial p}{\partial x},
			\end{equation}
		with 
		$$C_1 = -\frac{\Gamma(1-4\nu)}{\Gamma(1-5\nu)}.$$
		\end{coro} 
    
    \section{Planar random motions with space--varying velocity}
    
    We start our analysis from the damped wave equation with space--depending velocities as follows
    \begin{equation}
    \frac{\partial^2 p}{\partial t^2}+2\lambda \frac{\partial p}{\partial t}= c_1(x)\frac{\partial}{\partial x}c_1(x)\frac{\partial p}{\partial x}+c_2(y)\frac{\partial}{\partial y}c_2(y)\frac{\partial p}{\partial y}.
       \end{equation} 
    By taking the change of variables
    \begin{equation}
    \begin{cases}
    z =\int_0^x\frac{dx'}{c_1(x')}\\
    w = \int_0^y \frac{dy'}{c_2(y')}, \quad (x,y)\in \mathbb{R}^2,
    \end{cases}
    \end{equation}
    we obtain 
     \begin{equation}\label{4.3}
        \frac{\partial^2 p}{\partial t^2}+2\lambda \frac{\partial p}{\partial t}= \frac{\partial^2 p}{\partial z^2}+\frac{\partial^2 p}{\partial w^2}.
      \end{equation} 
     The absolutely continuous component of the distribution of the position $(X(t),Y(t))$ of the moving particle performing the planar
     motion described in the introduction satisfies \eqref{4.3} (see \cite{kole}).
     Therefore, returning to the original variables $(x,y)$ we are able to understand the role played by the variable velocity on the 
     model considered in \cite{kole}.
     The absolutely continuous component of the probability law is given by
     \begin{equation}\label{plane}
 	 p(x,y,t)=\frac{\lambda}{2\pi c_1(x)c_2(y)}\frac{\exp\bigg\{-\lambda t+\lambda\sqrt{t^2-\bigg|\displaystyle\int_0^x\frac{dx'}{c_1(x')}\bigg|^2-\bigg|\displaystyle\int_0^y\frac{dy'}{c_2(y')}\bigg|^2}\bigg\}}{\sqrt{t^2-\bigg|\displaystyle\int_0^x\frac{dx'}{c_1(x')}\bigg|^2-\bigg|\displaystyle\int_0^y\frac{dy'}{c_2(y')}\bigg|^2}},
     \end{equation}
     provided that both $c_1(x)$ and $c_2(y)$ are positive and such that $\displaystyle\int_0^x\frac{dx'}{c_1(x')}<\infty$ for all $x\in \mathbb{R}$ and
     $\displaystyle \int_0^y\frac{dy'}{c_2(y')}<\infty$ for all $y\in \mathbb{R}$, respectively.
     Therefore the support of $p(x,y,t)$ is given by the set
     \begin{equation}
     \mathcal{D}:=\bigg\{(x,y):\displaystyle\bigg|\int_0^x\frac{dx'}{c_1(x')}\bigg|^2+\displaystyle \bigg|\int_0^y\frac{dy'}{c_2(y')}\bigg|^2<t^2\bigg\}.
     \end{equation}
     The set $\mathcal{D}$ is therefore a \textit{deformation} of the circle representing the support of $(X(t),Y(t))$ in the case of constant 
     velocity.
     From formula \eqref{plane}, we can extract the conditional distribution of this class of generalized planar random motions. Since
     \begin{align}
     P\{X(t)\in dx, Y(t)\in dy\} = \sum_{n=0}^{\infty}P\{X(t)\in dx, Y(t)
     \in dy|N(t)= n\}P\{N(t)= n\}dxdy,
     \end{align}
     where $P\{N(t)= n\}$ is the homogeneous Poisson distribution of rate $\lambda$, we have that the conditional distribution is obviously given by
     \begin{equation}\label{conditional}
     P\{X(t)\in dx, Y(t) \in dy|N(t)= n\}= \frac{n}{2\pi t^n}\bigg[t^2-\bigg|\int_0^x\frac{dx'}{c_1(x')}\bigg|^2-\bigg|\int_0^y\frac{dy'}{c_2(y')}\bigg|^2\bigg]^{\frac{n}{2}-1}\frac{dxdy}{c_1(x)c_2(y)}
     \end{equation}
     The planar motion with space-varying velocity after $n$ changes of direction can be described as 
     \begin{equation}\label{rve}
     \begin{cases}
     X(t)= \displaystyle\sum_{j=1}^{n+1}\left(\int_{t_{j-1}}^{t_j} c_1(X(s))ds\right)\cos \theta_j\\
     Y(t)=\displaystyle\sum_{j=1}^{n+1}\left(\int_{t_{j-1}}^{t_j}c_2(Y(s))ds\right)\sin \theta_j,
     \end{cases}
     \end{equation}
     where $0= t_0<t_1< \dots<t_j< \dots<t_n<t_{n+1}= t$ are the epochs of the Poisson process and $\theta_j$ are the directions 
     of motion assumed at times $t_j$. The reader can ascertain that \eqref{rve} coincides with equation (12) of \cite{kole}
     in the case $c_1 = c_2 = const.$. Furthermore the conditional characteristic function of \eqref{rve} become
     \begin{align}
     \nonumber &\mathbb{E}\{e^{i\alpha X(t)+i\beta Y(t)}|N(t)=n\}  =\frac{n!}{(2\pi)^{n+1}t^n}\int_0^t dt_1\dots\int_{t_{n-1}}^t dt_n\int_0^{2\pi}d\theta_1\dots \int_0^{2\pi}
     d\theta_{n+1}\times \\
     \nonumber & \times \exp\bigg\{i\alpha \sum_{j=1}^{n+1}\left(\int_{t_{j-1}}^{t_j}c_1(X(s))ds\right)\cos\theta_j+i\beta\sum_{j=1}^{n+1}\left(\int_{t_{j-1}}^{t_j}c_2(Y(s))ds\right)\sin\theta_j\bigg\}\\
     \nonumber & = \frac{n!}{t^n}\int_0^t dt_1\dots \int_{t_{n-1}}^t dt_n \prod_{j=1}^{n+1} J_0\left(\sqrt{\alpha^2\left(\int_{t_{j-1}}^{t_j} c_1(X(s))ds\right)^2+\beta^2\left(\int_{t_{j-1}}^{t_j}c_2(Y(s)) ds\right)^2}\right).
     \end{align}
     By means of the change of variable
     \begin{equation}
     \begin{cases}
     du_j= c_1(X(s))ds\\
     dv_j = c_2(Y(s))ds,
     \end{cases}
     \end{equation}
     we can reproduce the calculations in the proof Theorem 1 in \cite{kole} in order to obtain the conditional distribution \eqref{conditional}. We can conclude that equation \eqref{plane} gives the absolutely continuous component of the probability law of the random vector $(X(t), Y(t))$ defined
     in \eqref{rve}.
     
     In order to understand the role of considering different velocities on both axes we consider a general domain that includes some interesting cases. It corresponds to taking the space-dependent velocities of the form $c_1(x)= \frac{|x|^\gamma}{c_1}$ and $c_2(x)=\frac{|x|^\beta}{c_2}$, with $\gamma, \beta < 1$ and $c_1, c_2>0$.
     With this choice, we obtain a family of probability laws concentrated in the domains of the form
     \begin{equation}
     \mathcal{D}_{\gamma, \beta}:=\bigg\{(x,y):\left(\frac{c_1 |x|^{1-\gamma}}{1-\gamma}\right)^2+\left(\frac{c_2 |y|^{1-\beta}}{1-\beta}\right)^2<t^2\bigg\}.
     \end{equation}
     \\
     This means that the boundary of the support of this family of probability law is given by a superellipse, also known as a Lam\'e curves including a wide class of geometrical figures like hypoellipses (for $\gamma =\beta<0$) and hyperellipses (for $\gamma =\beta>0$).
       We consider, in particular, two interesting cases.\\
          The first one, is the case in which $c_1(x)= c_1$ and
          $c_2(y)= c_2$ and $c_1\neq c_2\neq 0$. In this case the support of the probability law is clearly given by the ellipse:
          \begin{equation}
          \mathcal{D}_{0,0}:\bigg\{(x,y):c_1^2 |x|^2 +c_2^2 |y|^2<t^2\bigg\}.
          \end{equation}
          The second interesting case is given by the choice $c_1(x)= c_1 x^{2/3}$ and $c_2(y)= c_2 y^{2/3}$, leading to the compact support
          \begin{equation}
           \mathcal{D}_{2/3,2/3}:\bigg\{(x,y):9c_1^2 |x|^{2/3}+9c_2^2|y|^{2/3}<t^2\bigg\}.
          \end{equation}
          For $c_1=c_2 =1/3$ we obtain as boundary of $\mathcal{D}_{2/3,2/3}$, the astroid (see Figure 1). For $c_1\neq c_2\neq 1$ we have instead a squeezed astroid, possibly on both axes.
      \begin{figure}
                             \centering
                             \includegraphics[scale=.23]{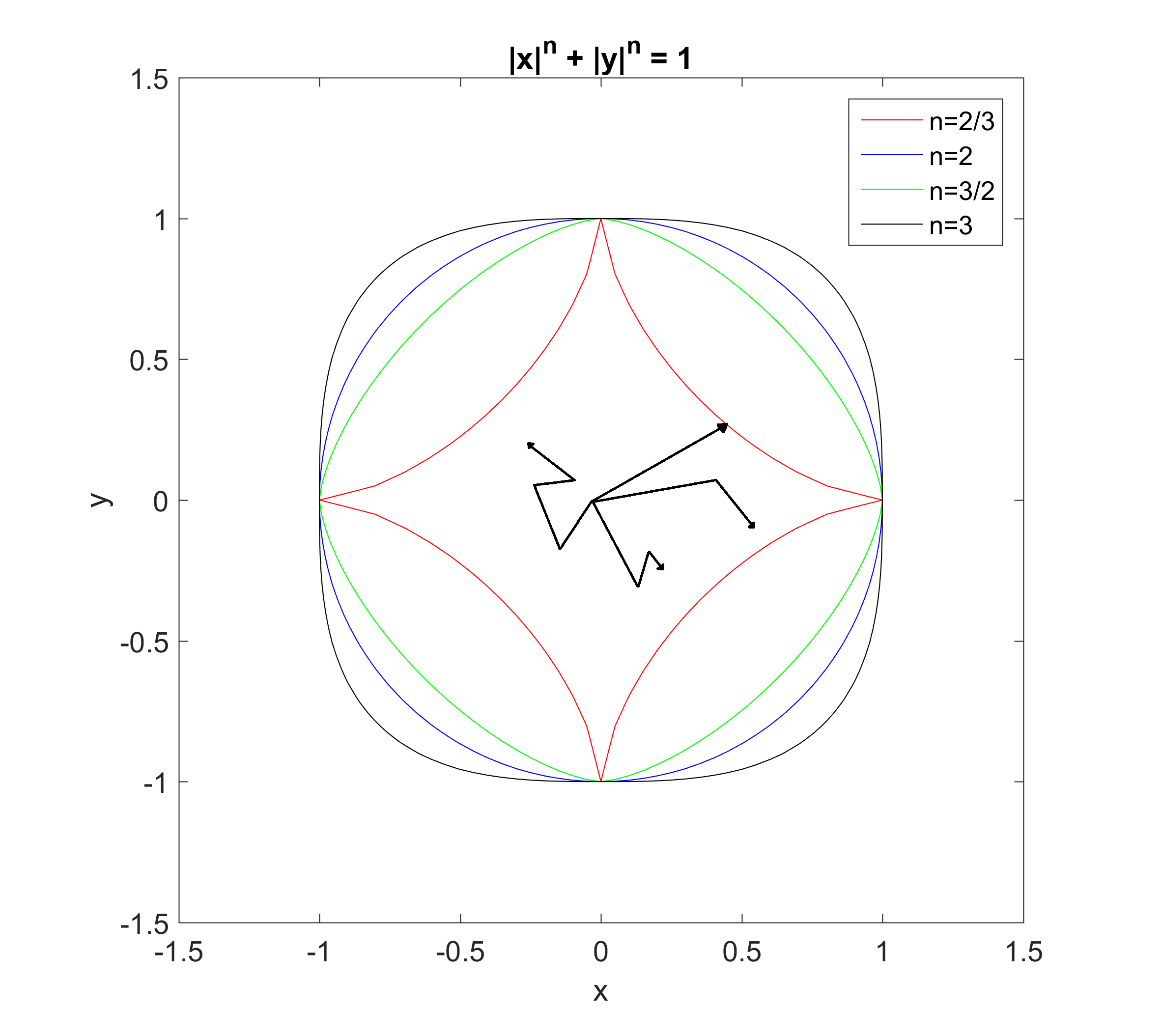}
                             \caption{We represent the boundary of the compact support $\mathcal{D}_{\gamma,\beta}$ of the family of probability laws \eqref{plane}, in the hypoelliptic ($n= 2/3$ leading to the astroid and $n= 3/2$), elliptic ($n=2$) and hyperelliptic case ($n= 3$). Here we assume that $c_1 = c_2 =1$ and $n = 2(1-\gamma)= 2(1-\beta)$. We show, as an example, four sample paths  with zero, two, three and four changes of direction, in the case when the random motion takes place in the astroid. 
                              } 
                             \label{figura0}
                  \end{figure}
     \newpage
                  
     Another interesting class of $d$-dimensional random motions at finite velocities is related to the EPD equation 
     \begin{equation}\label{EPs}
     				\frac{\partial^2 v}{\partial t^2}+\frac{2\alpha+d-1}{t}\frac{\partial v}{\partial t}= \sum_{j=1}^d
     				c_j(x_j)\frac{\partial}{\partial x_j}c_j(x_j)\frac{\partial v}{\partial x_j}.
     				\end{equation}
   In this case the probability law $p(x_1,x_2,\dots,x_d,t)$ of the particle moving in the $d$-dimensional space has the form 
     				  \begin{align}\label{fonos}
     				  &p(x_1,x_2, \dots, x_d, t) =  \frac{1}{\prod_{j=1}^d c_j(x_j)}\frac{\Gamma(\alpha+\frac{d}{2})}{\pi^{d/2}\Gamma(\alpha)t^{d+2\alpha-2}}\left(t^2-\sum_{j=1}^d
     				  \bigg|\int_0^{x_j}\frac{du_j}{c_j(u_j)}\bigg|^2\right)^{\alpha-1},\\ 
     				  \nonumber &\mbox{for $\displaystyle\sum_{j=1}^d
     				       				  \bigg|\int_0^{x_j}\frac{du_j}{c_j(u_j)}\bigg|^2<t^2$}
     				     \end{align}
     and represents a solution of \eqref{EPs}. The projection on the axes $x_1, \dots, x_d$ of the probability law \eqref{fonos} coincides with the one-dimensional motion dealt with in Section 3.
     
     A more general random motion in $\mathbb{R}^d$ with space-varying velocities leads to equation
     \begin{equation}
     \frac{\partial^2 p}{\partial t^2}+2\lambda(t)\frac{\partial p}{\partial t}=\sum_{j=1}^d c_j(x_1,\dots x_j, \dots, x_d) \frac{\partial}{\partial x_j}c_j(x_1,\dots x_j, \dots, x_d) \frac{\partial}{\partial x_j}p.
     \end{equation}
     This can be object of future research.

        \end{document}